\documentclass{amsart}

\usepackage[margin=1in]{geometry}
\usepackage{amsmath}
\usepackage{amsfonts}
\usepackage{amssymb}
\usepackage{amsthm}
\usepackage{amsrefs}
\usepackage{graphicx}

\newcommand{\ignore}[1]{}

\newcommand{\abs}[1]{\left\lvert {#1} \right\rvert}
\newcommand{\norm}[1]{\left\lVert {#1} \right\rVert}

\newcommand{\C}{{\mathbb{C}}}

\newcommand{\N}{{\mathbb{N}}}

\newcommand{\sP}{{\mathcal{P}}}

\newcommand{\bfe}{\mathbf{e}}
\newcommand{\init} {\operatorname{in}}

\newtheorem{thm}{Theorem}

\newtheorem{prop}{Proposition}

\newtheorem{cor}{Corollary}

\newtheorem{lemma}{Lemma}

\theoremstyle{definition}
\newtheorem{defn}{Definition}

\theoremstyle{remark}
\newtheorem{remark}{Remark}

\author{Dusty Grundmeier}
\thanks{The first author is partially supported by NSF RTG grant DMS-1045119.}

\author{Jennifer Halfpap}
\thanks{The second author was supported in part by NSF grant DMS 1200815.}

\title{An Application of Macaulay's Estimate to CR Geometry}

\begin{document}

\maketitle

\begin{abstract}
Several questions in CR geometry lead naturally to the study of bihomogeneous polynomials $r(z,\bar{z})$ on $\C^n \times \C^n$ for which $r(z,\bar{z})\norm{z}^{2d}=\norm{h(z)}^2$ for some natural number $d$ and a holomorphic polynomial mapping  $h=(h_1, \ldots,  h_K)$ from $\C^n$ to $\C^K$. When $r$ has this property for some $d$, one seeks relationships between $d$, $K$, and the signature and rank of the coefficient matrix of $r$. In this paper, we reformulate this basic question as a question about the growth of the Hilbert function of a homogeneous ideal in $\C[z_1,\ldots,z_n]$ and apply a well-known result of Macaulay to estimate some natural quantities.
\end{abstract}

\section{Introduction}
This paper is part of a continuing effort to explore the use of commutative algebra in the study of mapping problems in CR geometry. Suppose $r(z,\bar{z})$ is a bihomogeneous polynomial of bidegree $(m,m)$ on $\C^n \times \C^n$.  It is a well-known result of Quillen \cite{Quillen}, proved independently by Catlin and D'Angelo \cite{CD}, that if $r$ is positive on the unit sphere $S^{2n-1}$, then there exists a natural number $d$ such that
\begin{equation}\label{eq: QCD result}
r(z,\bar{z})\norm{z}^{2d}=\sum_{k=1}^K |h_k(z)|^2=\norm{h(z)}^{2}
\end{equation}
for holomorphic polynomials $h_k \colon \C^n \to \C$. See \cite{CDbundles} for an interpretation of this result in terms of holomorphic line bundles.
Recent investigations (\cite{DZ}, \cite{ToYeung}) have sought results linking the minimum $d$ for which $r(z, \bar{z})\norm{z}^{2d}$ is a squared norm to properties of $r$.  This minimum $d$ must depend on the size of $r$ on the sphere and not just on the dimension $n$ and the bidegree $(m,m)$, as the example
$$(|z_1|^2-|z_2|^2)^2+\varepsilon |z_1 z_2|^2 $$
shows; in this example, the minimum $d$ tends to infinity as $\varepsilon$ tends to $0$.

Our approach is algebraic; given an $r$ and $d$ for which
\eqref{eq: QCD result} holds, we give estimates on the signature of $r$ in terms of the dimension, the degree, and $d$. As in \cite{HL:sig}, we obtain as a corollary of our main theorem a lower bound on $D_d=\binom{d+n-1}{d}$ in terms of only the signature of $r$.  Questions about the relationships among these quantities arise naturally in CR geometry.

We reformulate our basic question as a question about the growth of the Hilbert function of a homogeneous ideal in the polynomial ring $\C[z_1, \ldots ,z_n]$, and we use a well-known estimate of Macaulay (Theorem \ref{thm: Macaulay} below). Ours is not the first paper to apply such results from commutative algebra to questions in CR geometry; Grundmeier, Lebl, and Vivas use a similar set of ideas in \cite{GLV} to prove a rigidity theorem for CR mappings of hyperquadrics.   Our goal is not primarily to obtain specific inequalities but rather to illustrate how a set of ideas from this area of algebra can be brought to bear on questions arising in CR geometry.

We would like to thank John D'Angelo for encouraging us to pursue an algebraic reformulation of such questions and for many helpful conversations.

\section{Definitions and a key lemma}

Let $r$ be a real-valued bihomogeneous polynomial of bidegree $(m,m)$ on $\C^n \times \C^n$; thus $$r(z,\bar{z})=\sum_{|\alpha|=|\beta|=m} c_{\alpha \beta} z^\alpha \bar{z}^\beta$$ for some Hermitian matrix $(c_{\alpha \beta})$.  Each such polynomial  has a holomorphic decomposition
\begin{equation}\label{eq: holo decomp}
r(z,\bar{z})=\norm{f(z)}^2 - \norm{g(z)}^2 =\sum_{j=1}^{P} \abs{f_j(z)}^2 -\sum_{k=1}^{N} \abs{g_k(z)}^2
\end{equation}
where the $f_j$, $g_k$, are holomorphic homogeneous polynomials of degree $m$ and the set $\{\,f_j,g_k:1\leq j \leq P,\; 1\leq k \leq N\,\}$ is linearly independent.
This decomposition is obtained by diagonalizing the coefficient matrix $(c_{\alpha \beta})$.
Thus, although the polynomials appearing in the holomorphic decomposition of a given $r$ are not uniquely determined, the signature pair $(P,N)$ and the rank $R=P+N$ are. Observe that $r$ has signature pair $(P,0)$ if and only if it is itself a squared norm.  The idea of studying real polynomials on $\C^n$ through Hermitian linear algebra has been used extensively by D'Angelo; see Chapter VI in \cite{JPD:Carus} and \cite{JPD(scv):93}.

Our results concern the dimensions of  various vector spaces naturally associated with $r$.  Let $\sP_{k}$ be the space of homogeneous polynomials of degree $k$ in $z_1,\ldots,z_n$. Let $D_{k}=\binom{k+n -1}{k}$, which is the dimension of this space.  For any mapping $h=(h_1,\ldots,h_K) \colon \C^n \to \C^K$ with components in $\sP_{k}$, let $V_h$ be the subspace of $\sP_{k}$ spanned by the components of $h$. If $g\colon \C^n \to \C^L$ is a second mapping, then the mappings $h \oplus g$ and $h \otimes g$ satisfy
$\norm{ h \oplus g}^2=\norm{h}^2+\norm{g}^2$
and
$\norm{h \otimes g}^2 = \norm{h}^2 \norm{g}^2$.  We can think of $z$ as the mapping whose components are the coordinate functions $z_1, \ldots , z_n$; Thus $z^{\otimes d}$ is the tensor product of this mapping with itself $d$ times.

The following lemma, which appears in a more general form in \cite{JPD:05}, is used heavily.
\begin{lemma}\label{prop: VG in VF}
Let $r$ be bihomogeneous of bidegree $(m,m)$ with
$r(z,\bar{z}) \norm{z}^{2d}=\norm{h(z)}^2$. Let $F = f\otimes z^{\otimes d}$ and $G=g \otimes z^{\otimes d}$.  Then $V_G \subseteq V_F$.
\end{lemma}
\begin{proof}
Note that
\begin{equation}\label{eq: r times power of norm}
r(z, \bar{z}) \norm{z}^{2d}=(\norm{f}^2 - \norm{g}^2)\norm{z}^{2d}=\norm{f\otimes z^{\otimes d}}^2 - \norm{g\otimes z^{\otimes d}}^2.
\end{equation}
Because $r(z,\bar{z})\norm{z}^{2d}=\norm{h(z)}^2$,
\begin{equation}
\norm{f \otimes z^{\otimes d}}^2= \norm{g \otimes z^{\otimes d}}^2+\norm{h}^2=\norm{(g \otimes z^{\otimes d}) \oplus h}^2,
\end{equation}
that is, $\norm{F(z)}^2=\norm{(G \oplus h)(z)}^2$.
Let $L$ be a natural number such that the rank of $F$ and the rank of $G \oplus h$ are at most $L$.  By adding identically zero components as needed, we obtain two mappings $F \oplus 0$ and $G\oplus h \oplus 0$ from $\C^n$ to $\C^L$ with the same squared norm.  Thus there exists a unitary transformation $U$ on $\C^L$ such that $U(F\oplus 0)= (G \oplus h \oplus 0)$.  It follows that
the components of the mapping $G$ are in the linear span of the components of the mapping $F$.
\end{proof}

To this point, for a mapping $h$ whose components are homogeneous of degree $m$, we have viewed $V_h$ as a vector subspace of $\sP_{m}$.  We change our perspective somewhat; we think of $S=\C[z_1,\ldots,z_n]$ as a graded ring, graded by degree. We think of $V_h$ as the homogeneous ideal generated by the components of $h$. $V_h$ is a graded $S$-module, with its degree $k$ component denoted by $(V_h)_k$.  The {\em Hilbert function} of $V_h$ is
$$H_{V_h}(k)=\dim_{\C} (V_h)_k. $$

We will make use of Macaulay's estimate on the growth of a homogeneous ideal.  Macaulay's result gives an upper bound for the Hilbert function for $S/V_h$ and hence a lower bound for the Hilbert function for $V_h$.
\begin{defn}
Let $c$ and $m$ be positive integers.  The $m$-th Macaulay representation of $c$ is the (unique) way of writing
\begin{equation}
c= \binom{k_m}{m} +\binom{k_{m-1}} {m-1} + \ldots + \binom{k_J} {J}
\end{equation}
where $k_m > k_{m-1}> \ldots > k_J \geq J > 0$. We also write
\begin{equation}
c^{<m>} = \binom{k_m + 1} {m+1} + \binom{k_{m-1} + 1}{ m } + \ldots + \binom{k_J + 1} {J + 1}.
\end{equation}
\end{defn}
See \cite{Green:gin} and \cite{GreenRestrictions} for a more extensive discussion of these ideas, including proofs of the uniqueness of the $m$-th Macaulay representation of a positive integer and of the elementary properties of the function $c \mapsto c^{<m>}$.  We need only two such properties:
\begin{lemma}\label{lemma: order preserved}
Let $b$, $c$ and $m$ be positive integers.
\begin{enumerate}

\item If
$b < c$, then $b^{<m>} <  c^{<m>}$.

\item For any $k \in \N$ with $k < c$, $(c-k)^{<m>} \leq c^{<m>} - k$.
\end{enumerate}
\end{lemma}

\begin{thm}[Macaulay's estimate on the growth of an ideal \cite{Mac1}] \label{thm: Macaulay}  Let $I$ be an ideal in $S=\C[z_1,\ldots , z_n]$ whose generators are homogeneous polynomials (not necessarily all of the same degree). Then
\begin{equation}\label{Macaulay's estimate}
H _{S/I}(k+1) \leq H_{S/I}(k)^{<k>}.
\end{equation}
\end{thm}
Elementary linear algebra shows that
\begin{equation}
H_{S/I}(k)+H_{I} (k)={k+n-1 \choose k}=D_{k}.
\end{equation}

\section{The main theorem and its corollaries}

In this section, we establish our main result concerning signatures of real polynomials on $\C^n$. See \cite{HL:sig} for a discussion of connections of this problem to CR geometry and for constructions of families of polynomials illustrating the extent to which some such inequalities are sharp. Additional comments on the sharpness of these inequalities appear at the end of the section.
\begin{thm}\label{thm: result on signature}
 Let $r$ be bihomogeneous of bidegree $(m,m)$ with signature $(P,N)$ and rank $R$. Suppose $r(z, \bar{z})\norm{z}^{2d}$ is a squared norm. Then
\begin{equation}\label{eq: first form of ineq in thm}
P \geq \frac{D_{m+d}-[[ H(m)^{<m>}]^{<m+1>} \cdots]^{<m+d-1>}}{D_d}
\end{equation}
for $H(m)=D_m - R$.
If the $m$-th Macaulay representation of $D_{m}- R$ is
$\sum_{j=J}^m \binom{k_j}{j}$, this is equivalent to
\begin{equation}
P \geq \frac{D_{m+d}-\sum_{j=J}^m \binom{k_j + d}{j+d}}{D_{d}}.
\end{equation}
\end{thm}
\begin{proof}
We obtain upper and lower bounds for $H_{V_f}(m+d)$, which equals $H_{V_{f \oplus g}}(m+d)$ since, by Proposition \ref{prop: VG in VF}, the components of $V_f$ and $V_{f\oplus g}$ in degree $m+d$ are the same.  Since $(V_f)_m$ has as a basis the $P$ generators $f_j$ of the ideal $V_f$, $(V_f)_{m+d}$ is spanned by the polynomials $z^\alpha f_j$ for $|\alpha|=d$. We thus have the upper bound
$$H_{V_f}(m+d)\leq D_{d} P .$$
We use Theorem \ref{thm: Macaulay} to obtain an upper bound on $H_{S/V_{f \oplus g}}(m+d)$.  To simplify notation, we drop the subscript on $H$ when discussing the Hilbert function for $S/V_{f \oplus g}$.
Since $H_{V_{f \oplus g}}(m)= R$,
\begin{equation}\label{eq: H for coord ring in deg m}
H(m)=D_{m}-R.
\end{equation}
Suppose the $m$-th Macaulay representation of $D_{m}- R$ is $\sum_{j=J}^m \binom{k_j}{j}$.
According to Macaulay's estimate,
\begin{equation}
H(m+1)\leq H(m)^{<m>}=\sum_{j=J}^m \binom{k_j+1}{j+1}.
\end{equation}
Applying the result again and using Lemma \ref{lemma: order preserved} gives
\begin{equation}
H(m+2) \leq H(m+1)^{<m+1>} \leq [H(m)^{<m>}]^{<m+1>}= \left(\sum_{j=J}^m \binom{k_j+1}{j+1} \right)^{<m+1>}.
\end{equation}
Since the expression in parentheses is already in the form of the $(m+1)$st Macaulay representation of an integer, the full expression on the right can be easily evaluated and we obtain:
\begin{equation}
H(m+2)\leq \sum_{j=J}^m \binom{k_j+2}{j+2}.
\end{equation}
Iterating gives
\begin{equation}
H(m+d)\leq [[H(m)^{<m>}]^{<m+1>} \cdots ]^{<m+d-1>} = \sum_{j=J}^m \binom{k_j+d}{j+d}
\end{equation}
Since $H_{V_{f \oplus g}}(m+d) = D_{m+d}-H(m+d)$, the result follows.
\end{proof}


We give two corollaries. In both cases, a simple direct proof also exists that does not require Macaulay's estimate.
We sketch both arguments for Corollary \ref{cor: full rank}. Furthermore, since we always use the upper bound
$$H_{V_f}(m+d) \leq D_{d} P $$
when $f$ has $P$ components, we need only describe how we obtain the lower bounds.
\begin{cor}\label{cor: first}
 Let $r$ be bihomogeneous of bidegree $(m, m)$ with signature pair $(P,N)$ and rank $R$.  If $r(z,\bar{z})\norm{z}^{2d}$ is a squared norm, then
\begin{equation}
N \leq (P-1)(D_d-1).
\end{equation}
\end{cor}
\begin{proof}
First, observe that
\begin{equation}
D_{m} = \sum_{k=0}^{m} \binom{k+n-2}{k}.
\end{equation}
This follows by counting the monomials of degree $m$ in $n$ variables in two ways.  On the one hand, this number is $D_{m}$. On the other hand, we may count the monomials of degree $m$ involving $x_1^{m-k}$ for $0 \leq k \leq m$, which gives $\binom{k+n-2}{k}$. Sum over $k$.

By Theorem \ref{thm: result on signature},
$$D_d P \geq D_{m+d}- [[[D_m - R]^{<m>}]^{<m+1>} \cdots ]^{<m+d-1>}.$$
By statement (ii) in Lemma \ref{lemma: order preserved} and the above expression for $D_m$,
\begin{eqnarray*}
\lefteqn{D_{m+d}-[[[D_m - R]^{<m>}]^{<m+1>} \cdots ]^{<m+d-1>}} &&\\
&\geq& D_{m+d}-[[[D_m-1]^{<m>}]^{<m+1>} \cdots]^{<m+d-1>} + R - 1\\
&=&\sum_{k=0}^{m+d} \binom{k+n-2}{k} - \left[ \left[ \left[\sum_{k=1}^{m} \binom{k+n-2}{k}\right]^{<m>}\right]^{<m+1>} \cdots \right]^{<m+d-1>} +R-1\\
&=&\sum_{k=0}^{m+d} \binom{k+n-2}{k} - \sum_{k=1}^{m} \binom{k+n-2+d}{k+d} +R-1\\
&=&\sum_{k=0}^d \binom{k+n-2}{k}+R-1\\
&=&D_d + R - 1.
\end{eqnarray*}
Thus
$$\frac{P}{R} \geq \frac{D_d + R - 1}{D_d R},$$
which is equivalent to
$$N \leq (P-1)(D_d - 1) .$$
\end{proof}
\begin{remark}
The inequality in Corollary \ref{cor: first} immediately implies $N \leq P(D_d - 1)$.
The latter appears with a different proof in Theorem 1.1(i) in \cite{HL:sig}.
\end{remark}

\begin{cor}\label{cor: full rank}
 Let $r$ be bihomogeneous of bidegree $(m,m)$ with $r(z,\bar{z}) \norm{z}^{2d}$ a squared norm  and $R=D_{m}$  (which is the largest rank possible for this $m$ and $n$).  Then
\begin{equation}\label{eq: n and m}
P \geq \frac{D_{m+d}}{D_{d}}.
\end{equation}
\end{cor}
\begin{remark}
Suppose $n=2$.  Then \eqref{eq: n and m} gives $P/R \geq (m+d+1)/(d+1)(m+1)$.  See \cite{HL:sig} for an example of a polynomial in two variables for which this ratio is achieved.
\end{remark}
\begin{proof}
To obtain a proof using Theorem \ref{thm: result on signature}, we need only observe that $D_{m}-R=0$ and so the numerator of \eqref{eq: first form of ineq in thm} reduces to $D_{m+d}$.

Next we give a direct proof for the lower bound on
$H_{V_f}(m+d)=H_{V_{f \oplus g}}(m+d)$. If the rank of $r$ is $D_{m}=\dim \sP_{m}$, then in fact the components of $f \oplus g$ span $\sP_{m}$.  Thus the components of $F \oplus  G$ span $\sP_{m+d}$, i.e.,
$$H_{V_{f \oplus g}}(m+d) = D_{m+d}.$$ Then
$$P \geq \frac{D_{m+d}}{D_{d}}.$$
\end{proof}

We examine our method of proof more closely:  
We obtain  the upper bound $H_{V_f}(m+d) \leq D_d P$ because the $D_d P$ functions $ z^\alpha f_j $ for $ 1\leq j\leq P$ and $ |\alpha|=d $ span $(V_f)_{m+d}$. We thus have equality if and only if these functions are linearly independent.
On the other hand, we obtain our lower bound for $H_{V_{f\oplus g}}(m+d)$ using Macaulay's theorem.  It is considerably more difficult to say when equality holds here. (See Green \cite{Green:gin} for a more detailed discussion.)  We make a definition:
\begin{defn}
Let $I$ be a monomial ideal. $I$ is a {\bf lex segment ideal} in degree $s$ if $I_s$ is spanned by the first $\dim (I_s)$ monomials of degree $s$ in lexicographic order.
\end{defn}
When the $R$ generators of a monomial ideal $I$ are all of the same degree $m$, $I$ is a lex segment ideal in all degrees if and only if the generators are the first $R$ monomials of degree $m$ in lexicographic order.
The following results are known.
\begin{prop}[Macaulay]
If $I$ is a lex segment ideal in degree $s$ and has no generators in degree $s+1$, $H_{S/I}( s+1)=H_{S/I}(s)^{<s>}$.
\end{prop}
\begin{thm}[Gotzmann's Persistence Theorem] Let $I$ be a homogeneous ideal generated in degree $\leq s$.  If $H_{S/I}(s+1)=H_{S/I}(s)^{<s>}$, then $H_{S/I}(k+1)=H_{S/I}(k)^{<k>} $
for all $k \geq s$.
\end{thm}

With lexicographic order on the monomials and for a homogeneous ideal $I$, let $\init(I)$ be its initial ideal.  This is a monomial ideal.
Then $I$ and $\init(I)$ have the same Hilbert function in all degrees.
We immediately obtain the following proposition describing a hypothesis under which we obtain an equality upon application of Macaulay's theorem.
\begin{prop} Let $r=\norm{f}^2 - \norm{g}^2$ be as in Theorem \ref{thm: result on signature}.
Suppose $\init(V_{f \oplus g})$ is a lex segment ideal.  Then
\begin{equation}
H_{V_{f \oplus g}}(m+d)=D_{m+d}-[[H(m)^{<m>}]^{<m+1>}\ldots]^{<m+d-1>}.
\end{equation}
\end{prop}

We now consider whether Corollaries 1 and 2 are sharp for all $n$, $d$, and $m$. In Corollary 2, $H_{V_{f \oplus g}}(m+d)=D_{m+d}$, i.e., we have an equality rather than merely a lower bound.  Thus we ask whether we can have equality in the upper bound
$H_{V_{f \oplus g}}(m+d) \leq D_d P$.
In the monomial case the question is whether there are $D_{m+d}/D_d$ monomials $x^{\beta}$ such that the $D_{m+d}$ monomials $x^{\alpha+\beta}$ for $|\alpha|=d$ are distinct.
When $n=2$, such a collection of monomials sometimes exists. Indeed, take $d$ a natural number and consider $m=k(d+1)$ for some natural number $k$. Then the $k+1$ monomials
$x^{m-j(d+1)}y^{j(d+1)}$, $ 0 \leq j \leq k $
have this property.
On the other hand, for $n>2$, even when $d=1$, it is not possible to find such a collection of monomials.



Next we consider Corollary 1, which implies
\begin{equation}
\frac{N}{P} \leq D_d -1.
\end{equation}
In \cite{HL:sig}, it is shown that this inequality is essentially sharp for $d=1$ for all $n$.  This is not inconsistent with the above discussion; the family of polynomials constructed there has several interesting properties.
\begin{enumerate}
\item $V_{f \oplus g}$ is a monomial ideal and includes all monomials of degree $m$ in $n$ variables.  Thus $H_{V_{f \oplus g}} (m+d) = D_{m+d}$.

\item Consider the set $B$ of all components $x^\beta$ of $f$ for which $\beta_j \geq 1$ for all $1 \leq j \leq n$. The monomials $x^{\beta+\bfe^k}$ for $x^\beta \in A$ and $1\leq k \leq n$ are all distinct.

\end{enumerate}
Thus although there are components of $f$ giving rise to the same monomial of degree $m+1$, the proportion of such components of $f$ goes to zero as $m$ increases.

\section{Polynomials with small rank}

Theorem \ref{thm: result on signature} and its corollaries give lower bounds on $P/R$, and hence upper bounds on $N/P$. In a sense, these results say which signatures are possible if $r(z,\bar{z})\norm{z}^{2d}$ is a squared norm and $r$ has large rank; in \cite{HL:sig}, those polynomials for which $P/R$ is close to $1/D_d$ have large degree and rank. By contrast, our next two propositions explore the situation in which $r(z,\bar{z}) \norm{z}^{2d}$ is a squared norm but the rank of $r$ and the ratio $N/P$ are small.

\begin{prop}\label{prop: ranks not possible} If $r$ is the squared norm of a holomorphic homogeneous polynomial mapping, then either the rank $\rho$ of $r(z,\bar{z})\norm{z}^{2d}$ satisfies
    \begin{equation}\label{eq: rank must satisfy}
    n P - \frac{P(P-1)}{2}  \leq \rho \leq nP
    \end{equation}
    for some $P \in \N$ with $P<n$, or  $\rho \geq n(n+1)/2$.
\end{prop}
\begin{remark}
In other words, for a fixed dimension $n$, if $r$ is a squared norm, not all ranks of $r(z,\bar{z})\norm{z}^{2d}$ are possible; the smallest possible (non-zero) rank is $n$, and the next smallest are $2n-1$ and $2n$.  For example, when $n=3$, ranks $1$, $2$, and $4$ are not possible.
\end{remark}
\begin{proof} If $r$ is a squared norm, so is $\tilde{r}(z,\bar{z})=r(z,\bar{z})\norm{z}^{2(d-1)}$.  Thus by replacing $r$ by $\tilde{r}$ if necessary, it suffices to prove the result for $d=1$.

Suppose $r$ is the squared norm of the homogeneous holomorphic mapping $f=(f_1,\ldots,f_P)$ so that
\begin{equation}
H_{V_f}(m+1)\leq n P.
\end{equation}
We use Macaulay's estimate to obtain a lower bound.

Suppose first that $P < n$.
Then
\begin{eqnarray*}
\left(D_{m} - P\right)^{<m>}&=& \left( \sum_{k=1}^m \binom{k+ n-2}{k} +1- P\right)^{<m>}\\
&=&\left(\sum_{k=2}^m \binom{k+n-2}{k} + \binom{n- P}{1} \right)^{<m>}\\
&=&\sum_{k=2}^{m} \binom{k+n-1}{k+1}+\binom{n-P+1 }{ 2}\\
&=&\sum_{k=1}^{m+1}\binom{k+n-2}{k}+\binom{n-P+1 }{ 2} - (n-1)-\binom{n}{ 2}\\
&=&D_{m+1}-\left(\frac{n^2+n}{2}-\frac{(n-P+1)(n-P)}{2}\right).
\end{eqnarray*}
We obtain the lower bound
\begin{equation}
H_{V_f}(m+1) \geq \frac{n^2+n}{2}-\frac{(n-P+1)(n-P)}{2}=nP-\frac{P(P-1)}{2}.
\end{equation}
Thus for $P<n$,
$$nP - \frac{P(P-1)}{2} \leq \rho \leq nP.$$

Next suppose $ P \geq n$. Then
$$H_{S/V_f}(m+1) \leq (H_{S/V_f}(m))^{<m>} \leq (D_{m}-n)^{<m>}=D_{m+1}-\frac{n(n+1)}{2},$$
and $H_{V_f}(m+1) \geq n(n+1)/2$.
\end{proof}

Finally, we consider $r$ of signature $(P,1)$. In \cite{HL:sig}, the authors prove that if, in the holomorphic decomposition of such an $r$, the components of $f$ and $g$ are monomials, then $P \geq n$. Furthermore, it is clear that signature pair $(n,1)$ is possible.

We consider this question in the general situation in which the components of $f$ and $g$ are arbitrary homogeneous polynomials of degree $m$.  We use Macaulay's estimate to show that certain values for $(V_f)_k$ and $(V_{f \oplus g})_k$ are not possible if these spaces are known to agree for $k \geq m+1$.
\begin{prop}\label{prop: (P,1)}
 Let $r$ be bihomogeneous of bidegree $(m,m)$ with $r(z, \bar{z}) \norm{z}^2$ a squared norm.
If $N=1$, then
\begin{equation}
P^2+P \geq 2 n.
\end{equation}
\end{prop}
\begin{proof}
 Since $H_{V_f}(m)=P$ and $H_{V_{f \oplus g}}(m)=P + 1$, $H_{S/V_{f \oplus g}}(m)=D_{m} - (P +1)$.

We know that it is possible to take $P=n$.  Thus our concern in this proposition is whether the Macaulay estimates themselves rule out the possibility of smaller $ P $.  Thus we consider $P<n$ and compare upper and lower bounds for $H_{ S /V_{f \oplus g}}(m+1)$. We expect to see that if $P$ is too small, the lower bound given by the Macaulay estimates will be larger than our known upper bound.

To begin, we need the $m$-th Macaulay representation of $D_{m} - ( P+1)$, where $ P < n$.  We begin as in the previous proposition, writing
\begin{eqnarray}
\left( D_{m} - (P+1)\right)^{<m>}&=& \left( \sum_{k=1}^m \binom{k+ n-2}{k} - P\right)^{<m>} \nonumber\\
&=&\left(\sum_{k=2}^m \binom{k+n-2}{ k} + \binom{n-1 - P}{1} \right)^{<m>} \label{eq: 2 cases needed}.
\end{eqnarray}
We consider two cases.  First, suppose $P=n-1$.  Then
\begin{eqnarray*}
\eqref{eq: 2 cases needed}&=&\sum_{k=2}^m \binom{k+n-1}{ k+1}\\
&=&\sum_{k=1}^{m+1}\binom{k+n-2}{k}-\binom{n}{2}-\binom{n-1}{1}\\
&=&D_{m+1}-\frac{n(n+1)}{2}\\
&=&D_{m+1}-\frac{(P+1)(P+2)}{2}.
\end{eqnarray*}
For the second case, suppose $P<n-1$.  We obtain
\begin{eqnarray*}
\eqref{eq: 2 cases needed}&=&\sum_{k=2}^{m} \binom{k+n-1}{ k+1}+\binom{n-P }{ 2}\\
&=&D_{m+1}-\left(n+nP-\frac{P^2+P}{2} \right)
\end{eqnarray*}
Since for $P=n-1$
$$n+nP-\frac{P^2+P}{2}=\frac{(P+1)(P+2)}{2} ,$$
for any $P<n$
we obtain the lower bound
\begin{equation}
H_{S/V_{f\oplus g}}(m+1) = n+nP-\frac{P^2+P}{2}.
\end{equation}
 As always, we trivially have the upper bound
 $$H_{V_{f \oplus g}}(m+1) \le n P  .$$
 Thus necessarily
 \begin{equation}
 n+nP-\frac{P(P+1)}2{} \leq nP.
 \end{equation}
 Thus we require
 $$   2n\leq P^2+P .$$
\end{proof}

The two propositions of this section illustrate what information about bihomogeneous polynomials can be gained through the use of this algebraic tool; rather than establishing the existence of bihomogeneous polynomials with certain algebraic properties, Macaulay's estimate allows us to show that certain scenarios are impossible because of necessary relationships between the dimensions of various vector spaces naturally associated with a polynomial.

\bibliographystyle{alpha}
\bibliography{references}

\begin{bibdiv}
\begin{biblist}

\bib{CD}{article}{
      author={Catlin, David~W.},
      author={D'Angelo, John~P.},
       title={Positivity conditions for bihomogeneous polynomials},
        date={1997},
     journal={Math. Res. Lett.},
      volume={4},
      number={4},
       pages={555\ndash 567},
}

\bib{CDbundles}{article}{
      author={Catlin, David~W.},
      author={D'Angelo, John~P.},
       title={An isometric imbedding theorem for holomorphic line bundles},
        date={1999},
     journal={Math. Res. Lett.},
      volume={6},
      number={1},
       pages={43\ndash 60},
}

\bib{JPD(scv):93}{book}{
      author={D'Angelo, John~P.},
       title={Several complex variables and the geometry of real
  hypersurfaces},
   publisher={{CRC} Press},
        date={1993},
}

\bib{JPD:Carus}{book}{
      author={D'Angelo, John~P.},
       title={Inequlaities from complex analysis},
      series={Carus Mathematical Monographs},
   publisher={MAA},
        date={2002},
}

\bib{JPD:05}{article}{
      author={D'Angelo, John~P.},
       title={Complex variables analogues of {H}ilbert's seventeenth problem},
        date={2005},
     journal={Internat. J. Math.},
      volume={16},
      number={6},
       pages={609\ndash 627},
}

\bib{DZ}{misc}{
      author={Drouot, Alexis},
      author={Zworski, Maciej},
       title={A quantitative version of {C}atlin-{D'A}ngelo-{Q}uillen theorem},
         how={arXiv:1205.3248 [math.CV]},
        date={2012},
}

\bib{GreenRestrictions}{inproceedings}{
      author={Green, Mark~L.},
       title={Restrictions of linear series to hyperplanes, and some results of
  {M}acaulay and {G}otzmann},
        date={1989},
   booktitle={Algebraic curves and projective geometry (trento, 1988)},
      series={Lecture Notes in Math.},
   publisher={Springer},
       pages={76\ndash 86},
}

\bib{Green:gin}{inproceedings}{
      author={Green, Mark~L.},
       title={Six lectures on commutative algebra},
        date={2010},
      series={Mod. Birkh{\"a}user Class.},
   publisher={Birkh{\"a}user Verlag},
       pages={119\ndash 186},
}

\bib{GLV}{misc}{
      author={Grundmeier, Dusty},
      author={Lebl, Ji\v{r}\'{i}},
      author={Vivas, Liz},
       title={Bounding the rank of hermitian forms and rigidity for {CR}
  mappings of hyperquadrics},
         how={arxiv:1110.4082v2[mathCV]},
        date={2011},
}

\bib{HL:sig}{misc}{
      author={Halfpap, Jennifer},
      author={Lebl, Ji\v{r}\'{i}},
       title={Signature pairs of positive polynomials},
         how={arxiv:1211:0997},
        date={2012},
}

\bib{Mac1}{article}{
      author={Macaulay, F.~S.},
       title={Some properties of enumeration in the theory of modular
  systems.},
        date={1927},
     journal={Proc. Lond. Math. Soc.},
      volume={26},
       pages={531\ndash 555},
}

\bib{Quillen}{article}{
      author={Quillen, Daniel~G.},
       title={On the representation of hermitian forms as sums of squares},
        date={1968},
     journal={Invent. Math.},
      volume={5},
       pages={237\ndash 242},
}

\bib{ToYeung}{article}{
      author={To, Wing-Keung},
      author={Yeung, Sai-Kee},
       title={Effective isometric embeddings for certain hermitian holomorphic
  line bundles},
        date={2006},
     journal={J. London Math. Soc.},
      volume={2},
      number={73},
       pages={607\ndash 624},
}

\end{biblist}
\end{bibdiv}

\end{document}